\documentclass[12pt]{amsart}

\usepackage{amsmath}
\usepackage{amssymb}
\usepackage{amscd}
\usepackage{color}

\begin{document}

\title{Extensions of Quasipolar Rings}

\author{Orhan G\"{u}rg\"{u}n}
\address{Orhan G\"{u}rg\"{u}n, Department of Mathematics, Ankara University, Turkey}
\email{orhangurgun@gmail.com}

\date{\empty}
\date{}
\newtheorem{thm}{Theorem}[section]
\newtheorem{lem}[thm]{Lemma}
\newtheorem{prop}[thm]{Proposition}
\newtheorem{cor}[thm]{Corollary}
\newtheorem{exs}[thm]{Examples}
\newtheorem{defn}[thm]{Definition}
\newtheorem{nota}{Notation}
\newtheorem{rem}[thm]{Remark}
\newtheorem{ex}[thm]{Example}

\maketitle

\begin{abstract}

In this paper we define and study quasipolar general rings (with
or without identity) and extend many of the basic results to the
wider class. We obtain some new characterizations of quasipolar
and strongly $\pi$-regular elements by using quasipolar general
rings. We see that quasipolar general rings lies between strongly
$\pi$-regular and strongly clean general rings. Consequently, we
prove that $R$ is pseudopolar if and only if $R$ is strongly
$\pi$-rad clean and quasipolar.
\\[+2mm]
\end{abstract}

\section{Introduction}

Throughout this article, by the term ring we mean an associative
ring with identity, and by a general ring we mean an associative
ring with or without identity. For clarity, $R$ and $S$ will
always denote rings, and general rings will be denoted $I$ or $A$.
The notation $U(R)$ denotes the group of units of $R$, $J(I)$
denotes the Jacobson radical of $I$ and $Nil(I)$ denotes the set
of all nilpotent elements of $I$. The \emph{commutant} and
\emph{double commutant} of an element $a$ in a ring $R$ are
defined by $comm_R(a)=\{x\in R~|~xa=ax\}$, $comm_R^2(a)=\{x\in
R~|~xy=yx~\mbox{for all}~y\in comm_R(a)\}$, respectively. If there
is no ambiguity, we simply use $comm(a)$ and $comm^2(a)$. If
$R^{qnil}=\{a\in R~|~1+ax\in U(R)~\mbox{for every}~x\in comm(a)\}$
and $a\in R^{qnil}$,  then $a$ is said to be \emph{quasinilpotent}
\cite{H}. Set $J^{\#}(R)=\{ x\in R~\mid~\exists ~n\in {\Bbb
N}~\mbox{such that}~x^n\in J(R)\}$. Clearly, $J(R)\subseteq
J^{\#}(R) \subseteq R^{qnil}$.

An element $a\in R$ is called \emph{quasipolar} provided that
there exists an idempotent $p\in comm^2(a)$ such that $a+p\in
U(R)$ and $ap\in R^{qnil}$. A ring $R$ is \emph{quasipolar} in
case every element in R is quasipolar. This concept ensued from
Banach algebra. Indeed, for a Banach algebra $R$ (see \cite[Page
251]{H0}),

$$a\in R^{qnil}\Leftrightarrow
\lim\limits_{n\to\infty}\parallel a^n\parallel^{\frac{1}{n}}=0.$$

\noindent Quasipolar rings are a very popular subject. Properties
of quasipolar rings were studied in \cite{D, H0, H, K, KP, YC}.

Ara \cite{Ara} defined and investigated the notion of an exchange
ring without identity. Then Chen and Chen \cite{CC} introduced the
concept of strongly $\pi$-regular general ring. In \cite{N3},
Nicholson and Zhou defined the notion of clean general ring. They
extend some of the basic results about clean rings to general
rings. In \cite{WC}, Wang and Chen defined the concept of strongly
clean general ring, and some properties about strongly clean rings
were extended. These works motivate us to define quasipolar
general rings. In this paper we see that every strongly
$\pi$-regular general ring is a quasipolar general ring and any
quasipolar general ring is a strongly clean general ring. We also
see that every (two-sided) ideal of an quasipolar ring is an
quasipolar general ring, but there exist quasipolar general rings
which are not ideals of quasipolar rings
(Example~\ref{important}). In particular, we prove that $a\in R$
is strongly $\pi$-regular if and only if there exist a strongly
regular $s\in R$ and $n\in Nil(R)$ such that $a=s+n$ and $sn=ns=0$
(Theorem~\ref{chen}), and $a\in R$ is quasipolar if and only if
there exist a strongly regular $s\in comm^2(a)$ and $q\in
R^{qnil}$ such that $a=s+q$ and $sq=qs=0$ (Corollary~\ref{cor3}).

An element $a$ of $R$ is \emph{(generalized) Drazin invertible}
(\cite{KP}) \cite{D, K} in case there is an element $b\in R$
satisfying $ab^2 = b$, $b\in comm^2(a)$ and ($a^2b-a\in R^{qnil}$)
$a^2b-a\in Nil(R)$. Such $b$, if it exists, is unique; it is
called a \emph{(generalized) Drazin inverse} of $a$. Koliha
\cite{K} showed that an element $a\in R$ is Drazin invertible if
and only if $a$ is strongly $\pi$-regular (\cite[Lemma 2.1]{K}).
Koliha and Patricio \cite{KP} proved that an element $a\in R$ is
generalized Drazin invertible if and only if $a$ is quasipolar
(\cite[Theorem 4.2]{KP}). With this in mind, we show that, for a
general ring $I$, $a\in I$ is quasipolar if and only if there is
an element $b\in I$ satisfying $ab^2 = b$, $b\in comm^2(a)$ and
$a^2b-a\in QN(I)$ (Theorem~\ref{teo0}), and $a\in I$ is strongly
$\pi$-regular if and only if there is an element $b\in I$
satisfying $ab^2 = b$, $b\in comm^2(a)$ and $a^2b-a\in Nil(I)$
(Theorem~\ref{teo4}).

Finally, we characterize a pseudopolar element of a ring, and we
give the relations among quasipolarity, strong $\pi$-rad cleanness
and pseudopolarity. It is shown that $R$ is pseudopolar if and
only if $R$ is strongly $\pi$-rad clean and quasipolar
(Theorem~\ref{orhan}).

\section{Quasipolar General Rings}

Let $I$ be a general ring with $p,q\in I$, we write $p\ast q=p + q
- pq$. Let $$Q(I)=\{q\in I~|~p\ast q=0=q\ast p~\mbox{for
some}~p\in I\}.$$

\noindent Note that $J(I)\subseteq Q(I)$. We define a set
$$QN(I)=\{q\in I~|~qx\in Q(I)~\mbox{for
every}~x\in comm(q)\}.$$

\noindent Clearly, $J(I)\subseteq Q(I)$ and $Nil(I)\subseteq QN(I)$.
If $R$ has an identity, then we have $Q(R)=\{q\in R~|~1-q\in U(R)\}$
and $QN(R)=R^{qnil}$. Further, if $a\in QN(I)$, then $a$ is also
said to be \emph{quasinilpotent}.

\begin{lem}\label{lem1} The following conditions are equivalent
for a ring $R$.

\begin{enumerate}
    \item $R$ is quasipolar.
    \item For all $a\in R$, $a+p\in Q(R)$ and $p^2=p\in
    comm^2(a)$ where $a-ap\in QN(R)$.
\end{enumerate}
\end{lem}

\begin{proof} $(1)\Rightarrow(2)$ Let $a\in R$. Since $R$ is
quasipolar, there exists an idempotent $1-p\in R$ such that
$1-p\in comm^2(a)$, $-a+1-p=u\in U(R)$ and $a(1-p)=a-ap\in
R^{qnil}$. Then $a+p=q$, $p\in comm^2(a)$ where $q=1-u$ and $q\ast
r=0=r\ast q$ with $r=1-u^{-1}$. As $R^{qnil}=QN(R)$, $a-ap\in
QN(R)$.

$(2)\Rightarrow(1)$ If $-a+p=q$ where $p^2=p\in comm^2(a)$, $q\in
Q(R)$ and $a-ap\in QN(R)$, then  $a+1-p=1-q$ where $(1-p)^2=1-p\in
comm^2(a)$, $1-q\in U(R)$ and $a(1-p)\in R^{qnil}$.
\end{proof}

\begin{defn}\label{defn1}\rm{ An element $a$ in a general ring $I$ is called a
\emph{quasipolar element} if $a+p\in Q(I)$ and $p^2=p\in
comm^2(a)$ where $a-ap\in QN(I)$; and $I$ is called a
\emph{quasipolar general ring} if every element is quasipolar.}
\end{defn}

\begin{rem}\label{hoca}\rm{ If $I$ is isomorphic to a general ring $K$ by $f$,
then $a\in I$ is quasipolar if and only if $f(a)$ is quasipolar in
$K$.}
\end{rem}

\begin{ex}\rm{ Idempotents, nilpotents, quasinilpotents and quasiregular elements are
all quasipolar.}
\end{ex}

Recall that an element $a$ in a general ring $I$ is called a
\emph{strongly clean element} if it is the sum of an idempotent
and an element of $Q(I)$ which commute; and $I$ is called a
\emph{strongly clean general ring} if every element is strongly
clean (\cite{WC}). Hence, by  Definition~\ref{defn1}, quasipolar
elements (general rings) are strongly clean element (general
ring).

We begin with a very useful lemma.

\begin{lem}\label{commute} Let $I$ be a general ring. If $a\in Q(I)$ and $a\in
comm(b)$, then $c\in comm(b)$ where $a\ast c=0=c\ast a$.
\end{lem}

\begin{proof} Let $a\ast c=0=c\ast a$ and $ba=ab$. Then $a+c=ac=ca$. This implies that
$ba+bc-bca=0=ab+cb-cab$, and so

\[
\begin{aligned}
bc-bca=cb-cab.\\
 \end{aligned}
\text{\quad ~~~~~~~~~~~~~~~~~~~~~~~~~~~~~~~\quad}
\begin{gathered}
(2.1)\\
\end{gathered}\]

\noindent Multiplying $(2.1)$ by $c$ from the right, yields
$$bcc-bcac=cbc-cabc.$$

\noindent This gives $bca=cba=cab$ because $c-ac=-a$. This shows
that $bc=cb$. That is; $c\in comm(b)$, as desired.
\end{proof}

\begin{lem}\label{qregular} Let $I$ be a general ring. If $a\ast b=0$ and $c\ast a=0$,
then $b=c$.
\end{lem}

\begin{proof} Suppose that $a\ast b=0$ and $c\ast a$ for $a, b, c\in
I$. This gives $b=0\ast b=(c\ast a)\ast b=c\ast (a\ast b)=c\ast
0=c$, as desired.
\end{proof}

\begin{lem} Let $I$ be a general ring and assume that $a\in I$ is
quasinilpotent. Then $a, -a\in Q(I)$ and $-a\in I$ is
quasinilpotent. Further $QN(I)\subseteq Q(I)$.
\end{lem}

\begin{proof} Since $a\in QN(I)$ and $a\in comm(a)$, we get $a^2\in
Q(I)$. That is, there exists a $b\in R$ such that $a^2\ast
b=a^2+b-a^2b=0=b+a^2-ba^2=b\ast a^2$. This implies that $0=a^2\ast
b=[a\ast(-a)]\ast b=a\ast[(-a)\ast b]$ and $0=b\ast
a^2=b\ast[(-a)\ast a]=[b\ast (-a)]\ast a$, and so we have $a\in
Q(I)$ by Lemma~\ref{qregular}. Similarly, it can be shown that
$-a\in Q(I)$. On the other hand, we check easily that $-a\in QN(I)$.
If $a\in QN(I)$, then $a\in Q(I)$. Hence $QN(I)\subseteq Q(I)$. The
proof is completed.
\end{proof}

It is well known that $a\in R$ is generalized Drazin invertible if
and only if $a$ is quasipolar. We now extend this fact to
quasipolar general rings.

\begin{thm}\label{teo0} The following are equivalent
for $a\in I$.

\begin{enumerate}
    \item $a$ is quasipolar in $I$.
    \item There exists $b\in comm^2(a)$ such that $ab^2=b$ and $a^2b-a\in
    QN(I)$.
\end{enumerate}
\noindent In this case $b$ is unique.
\end{thm}

\begin{proof} $(1)\Rightarrow (2)$ Write $a+p=q\in Q(I)$ where $p^2=p\in
comm^2(a)$ and $a-ap\in QN(I)$, say $q\ast r=r\ast q=0$ where
$r\in I$. Then $r+q=rq=qr$. In view of Lemma~\ref{commute},
$rp=pr$ because $q\in Q(I)$ and $q\in comm(p)$. Set $b=rp-p$. It
is easy to verify that $p=ab$. Let $ax=xa$ for some $x\in I$.
Since $p\in comm^2(a)$, we have $xp=px$ and so $xq=qx$. Moreover,
as $r+q=rq=qr$, we see that

\[
\begin{aligned}
xr-xrq=rx-rxq.\\
 \end{aligned}
\text{\quad ~~~~~~~~~~~~~~~~~~~~~\quad}
\begin{gathered}
(2.2)\\
\end{gathered}\]

\noindent Multiplying $(2.2)$ by $r$ from the right, yields
$$xrr-xrqr=rxr-rxqr ~~\mbox{and so}~~ xrq=rxq=rqx.$$

\noindent This shows that $rx=xr$. That is, $r\in comm^2(a)$.
Hence we conclude that $b\in comm^2(a)$. Now we show that
$ab^2=b$ and $a^2b-a\in QN(I)$. We have

$
\begin{array}{ll}
  ab^2 & =(q-p)(rp-p)(rp-p)=(q-p)(r^2p-rp-rp+p)  \\
       & = qr^2p-qrp-qrp+qp-r^2p+rp+rp-p \\
       & = qr^2p-rp-qp-rp-qp+qp-r^2p+rp+rp-p \\
       & = qr^2p-r^2p-p-qp\\
       & = (qr^2-r^2-p-q)p\\
       & = (r^2+rq-r^2-p-q)p\\
       & = (r-p)p\\
       & = b.\\
\end{array}
$

\noindent Moreover,

$
\begin{array}{ll}
  a^2b-a & =(q-p)(q-p)(rp-p)-(q-p)\\
        & = (q^2-qp-qp+p)(rp-p)-q+p \\
        & = q^2rp-q^2p-qrp+qp-qrp+qp+rp-p-q+p\\
        & = q^2rp-q^2p-rp-qp+qp-rp-qp+qp+rp-q\\
        & = q^2rp-q^2p-rp-q\\
        & = qpr+q^2p-q^2p-rp-q\\
        & = rp+qp-rp-q\\
        & = qp-q\\
        & = ap-a\in QN(I).\\
\end{array}
$

\noindent Thus holds $(2)$, as required.

$(2)\Rightarrow (1)$ Set $p=ab$. Then $p\in comm^2(a)$, and
$p^2=abab=a^2b^2=a(ab^2)=ab=p$. Since $a-ap=a-aab=a-a^2b$ and
$a^2b-a\in QN(I)$, we have $a-ap\in QN(I)$. Now we show that
$a+p=a+ab\in Q(I)$. We observe that
$(a+ab)\ast(b+ab)=a+ab+b+ab-(a+ab)(b+ab)=a+ab+b+ab-ab-a^2b-b-ab=a-a^2b$.
As $a-a^2b\in QN(I)$, $(a-a^2b)\ast x=x\ast (a-a^2b)=0$ for some
$x\in I$. This implies that $(a+ab)\ast(b+ab)\ast x=0$ and $x\ast
(b+ab)\ast (a+ab)=0$. Further, $(b+ab)\ast x=0\ast (b+ab)\ast
x=\big(x\ast (b+ab)\ast (a+ab)\big)\ast (b+ab)\ast x=\big(x\ast
(b+ab)\big) \ast \big((a+ab)\ast (b+ab)\ast x\big)=x\ast
(b+ab)\ast 0=x\ast (b+ab)$. Then $(b+ab)\ast x\ast (a+ab)=x\ast
(b+ab)\ast (a+ab)=0$. So we have $a+ab=a+p=q\in Q(I)$. Hence $a\in
I$ is quasipolar. Moreover, as $q\in Q(I)$, there exists $r\in I$
such that $q\ast r=0=r\ast q$, and so $r+q=rq=qr$. As in the
preceding discussion, we see that $r\in comm^2(a)$. Thus
$r\ast(q\ast(b+p))=(r\ast
q)\ast(b+p)=0\ast(b+p)=b+p=r\ast(a-ap)=r+a-ap-ra+rap=r+q-p-pq+p-rq+rpq=rpq-pq=rp$.
Therefore $b=rp-p$.

To prove the uniqueness of $b$, assume that $c\in comm^2(a)$ so
that $ac^2=c$ and $a^2c-a\in QN(I)$. Then
$ac-acab=ac-a^2cb=a^2c^2-a^2cb=(a^2b-a)(b-c)$. Since $a^2b-a\in
QN(I)$ and $b-c\in comm(a^2b-a)$, we have $ac-a^2cb\in Q(I)$. This
gives that $ac=a^2cb$. Similarly, we show that $ab=a^2cb$, and so
$ab=ac$. Thus $b=rp-p=rab-ab=rac-ac=c$; that is, $b$ is unique.
Note that $b$ is unique if and only if $p$ is unique. We complete
the proof.
\end{proof}

\begin{cor}\label{cor0} Let $I$ be a general ring. If $a\in I$ is
quasipolar, then $-a$ is quasipolar in $I$.
\end{cor}

\begin{proof} It is clear from Theorem~\ref{teo0}.
\end{proof}

Recall that an element $a$ in a general ring $I$ is called
\emph{strongly $\pi$-regular} if there exist $n\in \Bbb N$ and
$x\in I$ such that $a^n=a^{n+1}x$ and $x\in comm(a)$ (see \cite{A,
CC, WC}). We now extend \cite[Lemma 2.1]{K} to quasipolar general
rings.

\begin{thm}\label{teo4} The following are equivalent
for $a\in I$.

\begin{enumerate}
  \item $a$ is strongly $\pi$-regular in $I$.
  \item There exists a $p^2=p\in comm^2(a)$ such that $a-ap\in Nil(I)$ and $a+p\in Q(I)$.
  \item There exists a $p^2=p\in comm(a)$ such that $a-ap\in Nil(I)$ and $a+p\in Q(I)$.
  \item There exists a $b\in comm^2(a)$ such that $ab^2=b$ and $a^2b-a\in
    Nil(I)$.
  \item There exists a $b\in comm(a)$ such that $ab^2=b$ and $a^2b-a\in
    Nil(I)$.
\end{enumerate}
\end{thm}

\begin{proof} $(1)\Rightarrow(2)$ Assume that $a\in I$ is strongly $\pi$-regular. Then there exist $n\in \Bbb N$ and $x\in I$ such that
$a^{n}=a^{n+1}x$ and $ax=xa$. It is easy to check that
$a^{n}x^{n}=x^{n}a^{n}=p=p^2\in I$. Since $a^{n}=a^{n}x^{n}a^{n}$,
we have $(a-ap)^{n}=0$, and so $a-ap\in Nil(I)$.

\textbf{Claim 1.} $p\in comm^2(a)$.

\emph{Proof.} Let $ay=ya$. This implies that
$py-pyp=a^{n}x^{n}y-a^{n}x^{n}yp=a^{n}x^{n}y-x^{n}ya^{n}p=a^{n}x^{n}y-x^{n}ya^{n}=a^{n}x^{n}y-a^{n}x^{n}y=0$
because $ax=xa$ and $a^{n}x^n=x^na^{n}$. So $py=pyp$. Similarly,
we see that $yp=pyp$. Then $py=yp$ and so $p\in comm^2(a)$.

The remaining proof is to show that $q=a+p$ is a quasiregular
element of $I$. Set $t=a+a^2+a^3+\cdots+a^{n-1}$ and
$r=tp-t+a^{n-1}b^np+p$. Hence

$
\begin{array}{ll}
  q\ast r & =a+p+tp-t+a^{n-1}b^np+p-\\
          & atp-at-p-ap-a^{n-1}b^np-p\\
          & = a+p+ap-a-a^np+a^np-p-ap\\
          & = 0.\\
\end{array}
$

\noindent Analogously, we have $r\ast q=0$. Thus holds $(2)$.

$(2)\Rightarrow(3)$ Clear by $comm^2(a)\subseteq comm(a)$.

$(3)\Rightarrow(4)$ Assume that $a+p=q\in Q(I)$ where $p^2=p\in
comm(a)$ and $a-ap\in Nil(I)$, say $q\ast r=r\ast q=0$ and
$(a-ap)^k=a^k-a^kp=0$ where $r\in I$ and $k\in\Bbb N$. By
Lemma~\ref{commute}, $rp=pr$ because $q\in Q(I)$ and $q\in
comm(p)$. Set $b=rp-p$ and let $ax=xa$ for some $x\in I$. Then we
have $ab=p=ba$, and so
$xp-pxp=xa^kb^k-pxa^kb^k=a^kxb^k-pa^kxb^k=(a^k-pa^k)xb^k=0$. That
is, $xp=pxp$. Analogously, we see that $px=pxp$. This gives
$xp=px$. So $p\in comm^2(a)$. Therefore an argument similar to the
proof of Theorem~\ref{teo0} shows that $b\in comm^2(a)$, $ab^2=b$
and $a^2b-a=ap-a\in Nil(I)$.

$(4)\Rightarrow(5)$ It is obvious.

$(5)\Rightarrow(1)$ Let $ab=p$. Since $ab^2=b$, we have $p=p^2$.
As $a^2b-a\in Nil(I)$, there exists $k\in \Bbb N$ such that
$(a^2b-a)^k=0$. This implies that $(a^2b-a)^k=a^kp-a^k=0$. Then
$a^k=a^kp=a^kab=a^{k+1}b$ and $b\in comm(a)$. Hence $a\in I$ is
strongly $\pi$-regular. So holds $(1)$.
\end{proof}

\begin{rem}\rm{ If an element $a$ of  a general ring $I$ is strongly
$\pi$-regular, then $b$ and $p$ in Theorem~\ref{teo4} are unique
(indeed, as in the proof of Theorem~\ref{teo0}, we see that $b$
and $p$ are unique).}
\end{rem}

By Theorem~\ref{teo4}, the following result is immediate.

\begin{cor} Any strongly $\pi$-regular element (general ring) is strongly
clean element (general ring).
\end{cor}

Recall that an element $a$ of  a general ring $I$ is
\emph{strongly regular} if $a = aba$ and $b\in comm(a)$ for some
$b\in I$. $I$ is \emph{strongly regular} if every element in $I$
is strongly regular.

\begin{lem}\label{strregular} Let $I$ be a general ring and $a\in
I$. Then the following are equivalent.

\begin{enumerate}
    \item $a$ is strongly regular in $I$.
    \item There exists a $b\in comm^2(a)$ such that
    $a=a^2b$.
\end{enumerate}
\end{lem}

\begin{proof} It is similar to the proof of \cite[Lemma 1]{A}.
\end{proof}

Interestingly, we have the following result.

\begin{thm}\label{chen} For an element $a$ in a general ring
$I$, the following are equivalent.

\begin{enumerate}
    \item $a$ is strongly $\pi$-regular in $I$.
    \item $a\in I$ can be written in the form $a=s+n$ where $s$ is strongly regular, $n$ is nilpotent and $sn=ns=0$.
\end{enumerate}
\end{thm}

\begin{proof} $(1)\Rightarrow(2)$ Suppose that $a\in I$ is
strongly $\pi$-regular. It is well known that $a$ is strongly
$\pi$-regular if and only if $a$ is pseudo-invertible; that is,
there exist $c\in I$ and $m\in \Bbb N$ such that $ac=ca$,
$a^m=a^{m+1}c$ and $c=c^2a$ (see \cite[Theorem 4]{D}). Set $s=aca$
and $n=a-aca$. Then $sn=ns=aca(a-aca)=0$ because $ac=ca$ and $ac$
is idempotent in $I$. It is easy to checked that $s=s^2c$ and so
$s$ is strongly regular in $I$. Write $ca=ac=e=e^2\in I$. Hence
$(a-aca)^m=(a-ae)^m=a^m-a^me=a^m-a^mac=a^m-a^{m+1}c=0$. Thus $n\in
I$ is nilpotent and so holds $(2)$.

$(2)\Rightarrow(1)$ Assume that $a=s+n$ where $s$ is strongly
regular, $n$ is nilpotent and $sn=ns=0$. Since $n$ is nilpotent,
there exists $k\in \Bbb N$ such that $n^k=0$. As $s$ is strongly
regular, there exists $x\in I$ such that $s=s^2x$ and $x\in
comm^2(s)$ by Lemma~\ref{strregular}. Then it is easy to see that
$a^k=(s+n)^k=s^k$ and $a^{k+1}=(s+n)^k=s^{k+1}$ because $sn=ns=0$.
This gives that $a^k=s^k=s^{k-1}s=s^{k-1}s^2x=s^{k+1}x=a^{k+1}x$.
Further, as $as=sa$ and $x\in comm^2(s)$, we have $ax=xa$. Hence
$a$ is strongly $\pi$-regular in $I$.
\end{proof}

The following result is well known for a ring (see \cite{A}).

\begin{cor} If an element $a$ in a general ring $I$ is strongly
$\pi$-regular, then $a^k$ is strongly regular for some $k\in \Bbb
N$.
\end{cor}

 A new characterization of a quasipolar element in general ring is given as
 follows.

\begin{thm}\label{new} For an element $a$ in a general ring
$I$, the following are equivalent.

\begin{enumerate}
    \item $a$ is quasipolar in $I$.
    \item $a\in I$ can be written in the form $a=s+q$ where $s$ is strongly regular, $s\in comm^2(a)$, $q\in QN(I)$ and $sq=qs=0$.
\end{enumerate}
\end{thm}

\begin{proof} $(1)\Rightarrow(2)$ Assume that $a\in I$ is
quasipolar. By Theorem~\ref{teo0}, there exists $b\in comm^2(a)$
such that $ab^2=b$ and $a^2b-a\in QN(I)$. Set $s=a^2b$ and
$q=a-a^2b$. Further we have $s\in comm^2(a)$ and
$sq=qs=a^2b(a-a^2b)=0$ because $ab=ba$ and $ab$ is idempotent in
$I$. It is easy to see that $s=s^2b$ and so $s\in I$ is strongly
regular.

$(2)\Rightarrow(1)$ Suppose that $a=s+q$ where $s$ is strongly
regular, $s\in comm^2(a)$, $q\in QN(I)$ and $sq=qs=0$. Since $s$
is strongly regular, there exists $y\in comm^2(s)$ such that
$s=s^2y$ by Lemma~\ref{strregular}. Then we have $sy=ys$ is an
idempotent and $yq=qy$. Hence $a+sy=s+sy+q=(s+sy)\ast q=q\ast
(s+sy)$ and $(s+sy)\ast (y^2s+sy)=(y^2s+sy)\ast (s+sy)=0$. This
implies that $(a+sy)\ast (y^2s+sy)=(y^2s+sy)\ast (a+sy)=(s+sy)\ast
q\ast (y^2s+sy)=(y^2s+sy)\ast (s+sy)\ast q=q$. As $q\in Q(I)$, it
can be checked that $a+sy\in Q(I)$. Further
$a-asy=s+q-s^2y-qsy=q\in QN(I)$ and $sy\in comm^2(a)$. Thus $a\in
I$ is quasipolar. So holds $(1)$.
\end{proof}

The following result is a direct consequence of Theorem~\ref{new}.

\begin{cor}\label{cor3} Let $R$ be a ring and let $a\in R$. Then the following are equivalent.

\begin{enumerate}
    \item $a$ is quasipolar.
    \item $a=s+q$ where $s$ is strongly regular, $s\in comm^2(a)$, $q\in R^{qnil}$ and $sq=qs=0$.
\end{enumerate}
\end{cor}

\begin{prop}\label{cor1} A general ring $I$ is strongly regular if and only if $I$ is quasipolar and $QN(I)=0$.
\end{prop}

\begin{proof} Assume that $I$ is strongly regular. Then $I$ is strongly $\pi$-regular
and so $I$ is quasipolar by Theorem~\ref{teo4}. Let $a\in QN(I)$.
By hypothesis, $a = aba$ and $b\in comm(a)$ for some $b\in I$.
Since $ab=ba$, we have $ab\in Q(I)$. This implies that $ab=0$ and
so $a=0$. Hence $QN(I)=0$. Conversely, let $a\in I$. Since
$QN(I)=0$, $a$ is strongly regular by Theorem~\ref{new}.
\end{proof}

The following result follows from Proposition~\ref{cor1}.

\begin{cor}\cite[Theorem 2.4]{JCJC} Let $R$ be a ring. Then $R$ is strongly regular if and
only if $R$ is quasipolar and $R^{qnil}=0$.
\end{cor}

\begin{rem}\label{rem1}\rm{ $(1)$ In Proposition~\ref{cor1}, it was proved that if $a\in QN(I)$ and $a$ is
strongly regular, then $a=0$.

$(2)$ If $a$ is strongly regular, then $a^k$ is strongly regular
for any $k\in \Bbb N$.

$(3)$ If $a\in QN(I)$ and $a^k$ is strongly regular for some $k\in
\Bbb N$, then $a\in Nil(I)$.}
\end{rem}

\begin{prop}\label{cor2} A general ring $I$ is strongly
$\pi$-regular if and only if $I$ is quasipolar and $QN(I)\subseteq
Nil(I)$.
\end{prop}

\begin{proof} Assume that $I$ is strongly $\pi$-regular. Then, by Theorem~\ref{teo4}, $I$ is quasipolar
because $Nil(I)\subseteq QN(I)$. Let $a\in QN(I)$. As $I$ is
strongly $\pi$-regular, by Theorem~\ref{chen}, $a=s+n$ where $s$
is strongly regular, $n$ is nilpotent and $sn=ns=0$. Since $n$ is
nilpotent, there exists $k\in \Bbb N$ such that $n^k=0$. Hence we
have $a^k=s^k$. As $s^k$ is strongly regular and $a\in QN(I)$, by
Remark~\ref{rem1}, we see that $a\in Nil(I)$. Thus $QN(I)\subseteq
Nil(I)$. Conversely, suppose that $I$ is quasipolar and
$QN(I)\subseteq Nil(I)$. In view of Theorem~\ref{new} and
Theorem~\ref{chen}, $I$ is strongly $\pi$-regular.
\end{proof}

The following result is a direct consequence of
Proposition~\ref{cor2}.

\begin{cor}\cite[Theorem 2.6]{JCJC} Let $R$ be a ring. Then $R$ is strongly $\pi$-regular if and
only if $R$ is quasipolar and $R^{qnil}\subseteq Nil(R)$.
\end{cor}

An element $a$ of a ring $R$ is called \emph{semiregular} if there
exists $b\in R$ with $bab = b$ and $a-aba\in J(R)$. A ring is a
\emph{semiregular ring} if each of its elements is semiregular
(\cite[Proposition 2.2]{N0}).

We give a simple proof of the \cite[Theorem 3.2]{YC}.

\begin{thm}\label{semi}  Let $R$ be a ring. If $R$ is quasipolar and $R^{qnil}\subseteq J(R)$, then
$R$ is semiregular. The converse holds if $R$ is abelian.
\end{thm}

\begin{proof} Assume that $R$ is a quasipolar ring and $R^{qnil}\subseteq
J(R)$. Then we have $J(R)=R^{qnil}$. In view of
Corollary~\ref{cor3}, $R/J(R)$ is strongly regular. As $R$ is
quasipolar, $R$ is strongly clean and so idempotents lift modulo
$J(R)$. Then $R$ is semiregular by \cite[Theorem 2.9]{N0}.
Conversely, let $a\in R$. Then there exists $b\in R$ with $bab =
b$ and $a-aba\in J(R)$. Write $a=aba+(a-aba)$, say $s=aba$ and
$q=a-aba$. Since $a-aba\in J(R)\subseteq R^{qnil}$ and $R$ is
abelian, we see that $s\in comm^2(a)$, $q\in R^{qnil}$,
$s=aba=(aba)^2b=s^2b$ and $sq=qs=aba(a-aba)=a^2ba-a^2ba=0$. By
Corollary~\ref{cor3}, $a$ is quasipolar, and so $R$ is quasipolar. Take $x\in R^{qnil}$. By
assumption, there exists $y\in R$ with $yxy=y$ and $x-xyx\in
J(R)$. Note that $x\cdot 0=0$ and $x^2\cdot 0-x=-x\in R^{qnil}$. By Theorem~\ref{teo0}, we get $y=0$. This gives that $x\in J(R)$.
\end{proof}






\section{Extensions of quasipolar general rings}

Let $S$ be a ring and $I$ an $(S,S)$-bimodule which is a general
ring in which $(vw)s = v(ws)$, $(vs)w = v(sw)$ and $(sv)w = s(vw)$
hold for all $v,w\in I$ and $s\in S$. Then the {\em
ideal-extension} (it is also called the Dorroh extension) $I(S;I)$
of $S$ by $I$ is defined to be the additive abelian group $E(S;I)
= S\oplus I$ with multiplication $(s, v)(r,w) = (sr, sw + vr +
vw)$. In this case $I\lhd E(S;I)$, and $E(S;I)/I\cong S$. In
particular, $E(\Bbb Z;I)$ is the standard unitization of the
general ring $I$.

Clean general ideal-extensions are considered in \cite[Proposition
7]{N3}. Now we deal with quasipolar general ideal-extensions.

\begin{prop}\label{prop1} The following are equivalent for a
general ring $I$.

\begin{enumerate}
    \item $I$ is quasipolar.
    \item $(0,a)$ is quasipolar in $E(\Bbb Z;I)$ for all $a\in I$.
    \item There exists a ring $S$ such that $I=_{S}I_{S}$ and $(0,a)$
    is quasipolar in $E(S;I)$ for all $a\in I$.
\end{enumerate}
\end{prop}

\begin{proof} (1) $\Rightarrow$ (2) Let $a\in I$ and $R=E(\Bbb Z;I)$. By
Theorem~\ref{new}, we have $-a=s+q$ where $s\in I$ is strongly
regular, $q\in QN(I)$ and $sq=qs=0$. Write $(0,a)=(0,-s)+(0,-q)$.
Since $s$ is strongly regular, there exists $y\in comm^2(s)$ such
that $s=s^2y$ by Lemma~\ref{strregular}. This implies that
$(0,-s)=(0,-s)^2(0,-y)$ and $(0,-y)\in comm^2\big((0,-s)\big)$,
and so, by Lemma~\ref{strregular}, $(0,-s)$ is strongly regular in
$R$. Assume that $(x,y)\in comm\big((0,q)\big)$. Then we have
$x+y\in comm(q)$ and so $(x+y)q\in Q(I)$ because $q\in QN(I)$.
This gives $(1,0)+(x,y)(0,-q)=\big(1,-(x+y)q\big)\in U(R)$ (the
inverse is $(1,-t)$ where $(x+y)q\ast t=0=t\ast (x+y)q$). Hence
$(0,-q)\in R^{qnil}$. As $sq=qs=0$, we see that
$(0,-s)(0,-q)=(0,-q)(0,-s)=(0,0)$, and so $(0,a)\in R$ is
quasipolar by Corollary~\ref{cor3}.

(2) $\Rightarrow$ (3) It is clear with $S=\Bbb Z$.

(3) $\Rightarrow$ (1) Let $a\in I$ and $R=E(S;I)$. By $(3)$,
$(0,-a)+(e,p)=(e,p-a)$ where $(e,p)^2=(e,p)\in
comm^2\big((0,-a)\big)$, $(e,p-a)\in U(R)$ and
$(0,-a)(e,p)=\big(0,-a(e+p)\big)\in R^{qnil}$. Since
$(e,p)^2=(e,p)$, we have $e^2=e$ and $p=ep+pe+p^2$. This gives
that $e=1_S$ because $(e,p-a)\in U(R)$. So $-p$ is an idempotent
in $I$. As $(-1,a-p)\in U(R)$, there exists $q\in I$ such that
$q\ast (a-p)=0=(a-p)\ast q$. This implies that $a+(-p)\in Q(I)$.
If $ax=xa$, then we have $(0,x)\in comm\big((0,-a)\big)$ and so
$xp=px$ because $(1,p)\in comm^2\big((0,-a)\big)$. Hence $-p\in
comm^2(a)$. Now we show that $a+ap\in QN(I)$. Let
$x(a+ap)=(a+ap)x$. As $\big(0,-a(1_S+p)\big)\in R^{qnil}$, it
follows that $x(a+ap)\in Q(I)$. So $a\in I$ is quasipolar. The
proof is completed.
\end{proof}

\begin{thm}\label{teo2} Let $I$ be a quasipolar general ring and
$A\lhd I$. Then $A$ is quasipolar.
\end{thm}

\begin{proof} Let $R=E(\Bbb Z;I)$ and $a\in A$. By
Theorem~\ref{new}, we have $-a=s+q$ where $s\in I$ is strongly
regular, $s\in comm^2(a)$, $q\in QN(I)$ and $sq=qs=0$. Write
$(0,a)=(0,-s)+(0,-q)$. Since $s$ is strongly regular, there exists
$y\in comm^2(s)$ such that $s=s^2y$ by Lemma~\ref{strregular}.
This implies that $(0,-s)=(0,-s)^2(0,-y)$ and $(0,-y)\in
comm^2\big((0,-s)\big)$, and so, by Lemma~\ref{strregular},
$(0,-s)$ is strongly regular in $R$. Assume that $(m,n)\in
comm\big((0,q)\big)$. Then we have $x+y\in comm(q)$ and so
$(m+n)q\in Q(I)$ because $q\in QN(I)$. This gives
$(1,0)+(m,n)(0,-q)=\big(1,-(m+n)q\big)\in U(R)$ (the inverse is
$(1,-t)$ where $(m+n)q\ast t=0=t\ast (m+n)q$). Hence $(0,-q)\in
R^{qnil}$. As $sq=qs=0$, we see that
$(0,-s)(0,-q)=(0,-q)(0,-s)=(0,0)$. Let $(u,v)(0,a)=(0,a)(u,v)$.
Then $(u+v)\in comm(a)$ and so $(u+v)\in comm(s)$ since $s\in
comm^2(a)$. This proves $(u,v)(0,-s)=(0,-s)(u,v)$. That is,
$(0,-s)\in comm^2\big((0,a)\big)$. So $(0,a)\in R$ is quasipolar
by Corollary~\ref{cor3}. As $A\cong (0,A)\lhd R$, $A$ is
quasipolar by Proposition~\ref{prop1} and Remark~\ref{hoca}.
\end{proof}

This result shows that any ideal of a quasipolar general ring is a
quasipolar general ring. For example, every ideal of a strongly
$\pi$-regular ring or every nil ideal is a quasipolar general
ring. But the converse need not be true in general as the
following example shows.

Given a ring $R$, the set $I=\{(a,b)~|~a,b\in R\}$ becomes a
general ring (without identity) with addition defined
componentwise and
multiplication defined by $(a,b)(c,d)=(ac,ad)$. Then  $I\cong \left[%
\begin{array}{cc}
  R & R \\
  0 & 0 \\
\end{array}%
\right]=J$ where $J$ is a right ideal of $M_2(R)$.

\begin{ex}\label{important}\rm{ Consider the local ring $R=\Bbb
Z_{(2)}=\{\frac{m}{n}\in \Bbb Q~|~2\nmid n\}$ and $(a,b)\in I$. If
$a\in J(R)$, then it is easy to verify that $(a,b)\in J(I)$ and so
$(a,b)$ is quasipolar in $I$. If $a\notin J(R)$, then $a\in
1+J(R)$. So $(a,b)+(1,a^{-1}b)=(a+1,b+a^{-1}b)$} where
$(1,a^{-1}b)^2=(1,a^{-1}b)\in comm^2\big((a,b)\big)$ and
$(a+1,b+a^{-1}b)\in J(I)\subseteq Q(I)$. Further, since
$(a,b)-(a,b)(1,a^{-1}b)=(0,0)\in QN(I)$, $(a,b)$ is quasipolar in
$I$. Hence $I$ is a quasipolar general ring. On the other hand
$M_2(R)$ is not a quasipolar ring because $M_2(R)$ is not a
strongly clean ring (see \cite{WC2}).
\end{ex}

\begin{lem}\label{lem2} Let $e^2=e\in I$. Then $QN(eIe)=eIe\cap QN(I)$.
\end{lem}

\begin{proof} Let $a\in QN(eIe)$ and $ab=ba$ for some $b\in I$.
Then $a\cdot ebe=abe=bae=ba$ and $ebe\cdot a=eba=eab=ab$, so
$ebe\in comm(a)$. Since $a\in QN(eIe)$, we have $ab\ast x=0=x\ast
ab$ for some $x\in eIe$. Hence $a\in eIe\cap QN(I)$. This gives
that $QN(eIe)\subseteq eIe\cap QN(I)$.  Conversely, let $a\in
eIe\cap QN(I)$ and $aere=erea$ for some $ere\in eIe$. This implies
that $ae=ea=a$. Since $a\in QN(I)$, $are+y-arey=0=are+y-yare$ for
some $y\in I$. Then $are+eye-areye=0=are+eye-eyare$ and so $are\in
Q(eIe)$. Therefore $eIe\cap QN(I)\subseteq QN(eIe)$. We complete
the proof.
\end{proof}

\begin{thm}\label{teo3} Let $I$ be a quasipolar general ring with $e^2=e\in
I$. Then $eIe$ is quasipolar.
\end{thm}

\begin{proof} Let $a\in eIe$. Then there exists $p^2=p\in
comm^2(a)$ such that $a+p=q\in Q(I)$ and $a-ap\in QN(I)$. Since
$ae=ea$, we have $ep=pe$. This implies that $a+epe=eqe$ where
$epe^2=epe$ and $eqe\in Q(I)\cap eIe=Q(eIe)$. It is easy to see
that $epe\in comm^2(a)$ because $p^2=p\in comm^2(a)$. As $a-ap\in
QN(I)$, we have $a-ap=a-aep=a-aepe=a-ape=e(a-ap)e\in QN(I)\cap
eIe=QN(eIe)$ by Lemma~\ref{lem2}. Hence $eIe$ is quasipolar.
\end{proof}

\begin{cor}\label{cornerqsr} Let $R$ be a ring with $e^2=e\in R$. If $R$ is
quasipolar, then so is $eRe$.
\end{cor}

\section{Pseudopolar Elements}

An element $a$ of $R$ is \emph{pseudo Drazin invertible} in case
there exist $b\in R$ and $k\in \Bbb N$ satisfying $ab^2 = b$,
$b\in comm^2(a)$ and $(a-a^2b)^k\in J(R)$ . Such $b$, if it
exists, is unique; it is called a \emph{pseudo Drazin inverse} of
$a$. Wang and Chen \cite{WC3} showed that an element $a\in R$ is
pseudo Drazin invertible if and only if $a$ is pseudopolar; that
is, there exist $p\in R$ and $k\in \Bbb N$ such that $p^2=p\in
comm^2(a)$, $a+p\in U(R)$ and $a^kp\in J(R)$.

A characterization of pseudopolar elements can be given as
follows.

\begin{thm}\label{pseudo} Let $R$ be a ring and let $a\in R$. Then the following are equivalent.

\begin{enumerate}
    \item $a$ is pseudopolar.
    \item $a=s+q$ where $s$ is strongly regular, $s\in comm^2(a)$, $q\in J^{\#}(R)$ and $sq=qs=0$.
\end{enumerate}
\end{thm}

\begin{proof} $(1)\Rightarrow(2)$ Assume that $a\in R$ is
pseudopolar. Then there exist $b\in comm^2(a)$ and $k\in \Bbb N$
such that $ab^2=b$ and $(a-a^2b)^k\in J(R)$. Set $s=a^2b$ and
$q=a-a^2b$. This gives $s\in comm^2(a)$, $q\in J^{\#}(R)$ and
$sq=qs=a^2b(a-a^2b)=0$. It is easy to see that $s=s^2b$ and so
$s\in R$ is strongly regular.

$(2)\Rightarrow(1)$ Suppose that $a=s+q$ where $s$ is strongly
regular, $s\in comm^2(a)$, $q\in J^{\#}(R)$ and $sq=qs=0$. Since
$s$ is strongly regular, there exists $y\in comm^2(s)$ such that
$s=s^2y$ by Lemma~\ref{strregular}. Then we have $1-p=sy=ys$ is an
idempotent, $p\in comm^2(a)$ and $yq=qy$. As $q\in J^{\#}(R)$, we
see that $q^n\in J(R)$ and so $1+q\in U(R)$ for some $n\in \Bbb
N$. Hence $(a+p)(y^2s+p)=1+q\in U(R)$ and so $a+p\in U(R)$.
Moreover, $a^np=(s^n+q^n)(1-sy)=q^n\in J(R)$ because
$s^n=s^{n+1}y$. So holds $(1)$.
\end{proof}

Note that if $R$ is pseudopolar, then $R$ is quasipolar by
Theorem~\ref{pseudo} and Corollary~\ref{cor3}. Further, if $-a$ is
pseudopolar, then so is $a$ by Theorem~\ref{pseudo}.

Combining Theorem~\ref{teo4} and Theorem~\ref{pseudo}, we obtain
the following result.

\begin{cor}\cite[Theorem 2.1]{WC3} Let $R$ be a ring. Then $R$ is strongly $\pi$-regular if and
only if $R$ is pseudopolar and $J(R)$ is nil.
\end{cor}

We give a simple proof of the \cite[Theorem 2.4]{WC3}.

\begin{thm}  Let $R$ be a ring. If $R$ is pseudopolar and $J^{\#}(R)=J(R)$, then
$R$ is semiregular. The converse holds if $R$ is abelian.
\end{thm}

\begin{proof} Assume that $R$ is pseudopolar and $J^{\#}(R)=J(R)$.
According to Theorem~\ref{pseudo}, $R/J(R)$ is strongly regular.
Hence $R$ is semiregular by \cite[Theorem 2.9]{N0}. Conversely,
let $a\in R$. Then there exists $b\in R$ with $bab = b$ and
$a-aba\in J(R)$. Write $a=aba+(a-aba)$, say $s=aba$ and $q=a-aba$.
Since $a-aba\in J(R)\subseteq J^{\#}(R)$ and $R$ is abelian, we
see that $s\in comm^2(a)$, $q\in J^{\#}(R)$, $s=aba=(aba)^2b=s^2b$
and $sq=qs=aba(a-aba)=a^2ba-a^2ba=0$. By Theorem~\ref{pseudo}, $a$
is pseudopolar. In view of Theorem~\ref{semi}, we see that
$J^{\#}(R)=J(R)$.
\end{proof}

Recall that an element $a\in R$ is {\it strongly $\pi$-rad clean}
provided that there exists an idempotent $e\in R$ such that
$ae=ea$ and $a-e\in U(R)$ and $a^ne\in J(R)$ for some $n\in \Bbb
N$. A ring $R$ is {\it strongly $\pi$-rad clean} in case every
element in $R$ is strongly $\pi$-rad clean (see \cite{Di}). We now
give the relations among quasipolarity, strong $\pi$-rad cleanness
and pseudopolarity.

\begin{thm} \label{orhan} Let $R$ be a ring. Then $R$ is
pseudopolar if and only if $R$ is strongly $\pi$-rad clean and
quasipolar.
\end{thm}

\begin{proof} The `` only if " part is easy to see. So we have
only to prove the `` if " part. Let $a\in R$. Then there exists
$p^2=p\in comm^2(a)$ such that $a+p\in U(R)$ and $ap\in R^{qnil}$
since $R$ is quasipolar. Further, there exists $q\in comm(a)$ such
that $-a-q\in U(R)$ and $a^nq\in J(R)$ for some $n\in \Bbb N$
because $R$ is strongly $\pi$-rad clean. Since $a^nq\in J(R)$, we
have $aq\in R^{qnil}$. By \cite[Proposition 2.3]{KP}, we see that
$p=q$. Hence $a$ is pseudopolar, as desired.
\end{proof}

\begin{cor}\cite[Corollary 2.12]{WC3} Let $R$ be a ring with $e^2=e\in R$. If $R$ is
pseudopolar, then so is $eRe$.
\end{cor}

\begin{proof} Assume that $R$ is pseudopolar. Then $R$ is strongly $\pi$-rad clean and
quasipolar by Theorem~\ref{orhan}. In view of \cite[Corollary
4.2.2]{Di} and Corollary~\ref{cornerqsr}, $eRe$ is strongly
$\pi$-rad clean and quasipolar. Hence $eRe$ is pseudopolar again
by Theorem~\ref{orhan}.
\end{proof}

\begin{rem}\rm{ Let $S$ be a commutative ring and $R=M_2(S)$.
By \cite[Example 4.3]{WC3}, we have $J^{\#}(R)=R^{qnil}$. Hence,
by Theorem~\ref{pseudo} and Corollary~\ref{cor3}, $R$ is
quasipolar if and only if $R$ is pseudopolar. Further, if $S$ is
commutative local, then $R$ is pseudopolar if and only if $R$ is
quasipolar if and only if $R$ is strongly clean (by
\cite[Corollary 2.13]{GHH}) if and only if $R$ is strongly
$\pi$-rad clean (by \cite[Corollary 4.3.7]{Di}).}
\end{rem}


\begin{thebibliography}{99}

\bibitem{Ara} Ara, Pere Extensions of exchange rings. J. Algebra 197 (1997), no. 2, 409-423.

\bibitem{A} Azumaya, Gor\^{o} Strongly $\pi$-regular rings. J. Fac. Sci. Hokkaido Univ. Ser. I. 13, (1954), 34-39.

\bibitem{CC} Chen, Huanyin; Chen, Miaosen On strongly $\pi$-regular ideals. J. Pure Appl. Algebra 195 (2005), no. 1, 21-32.

\bibitem{JCJC} Cui, Jian; Chen, Jianlong Characterizations of quasipolar rings. Comm. Algebra 41 (2013), no. 11, 3207-3217.

\bibitem{Di} Diesl, Alexander James Classes of strongly clean rings. Thesis (Ph.D.)–University of California, Berkeley. 2006. 73 pp.

\bibitem{D} Drazin, M. P. Pseudo-inverses in associative rings and semigroups. Amer. Math. Monthly 65 (1958), no. 7, 506-514.

\bibitem{GHH} Gurgun, O.; Halicioglu S.; Harmanci, A. Quasipolarity of generalized matrix rings, http://arxiv.org/abs/1303.3173.

\bibitem{H0} Harte, Robin Invertibility and singularity for bounded linear operators.
Monographs and Textbooks in Pure and Applied Mathematics, 109.
Marcel Dekker, Inc., New York, 1988. xii+590 pp.

\bibitem{H} Harte, Robin On quasinilpotents in rings. Panamer. Math. J. 1 (1991), 10-16.

\bibitem{K} Koliha, J. J. A generalized Drazin inverse. Glasgow Math. J. 38 (1996), no. 3, 367-381.

\bibitem{KP} Koliha, J. J.; Patricio, Pedro Elements of rings with equal spectral idempotents. J. Aust. Math. Soc. 72 (2002), no. 1, 137-152.

\bibitem{N0}  Nicholson, W. K. Semiregular modules and rings. Canad. J. Math. 28 (1976), no. 5, 1105-1120.

\bibitem{N} Nicholson, W. K. Lifting idempotents and exchange rings. Trans. Amer. Math. Soc. 229 (1977), 269-278.

\bibitem{N2} Nicholson, W. K. Strongly clean rings and Fitting's lemma. Comm. Algebra 27 (1999), no. 8, 3583-3592.

\bibitem{N3} Nicholson, W. K.; Zhou, Y. Clean general rings. J. Algebra 291 (2005), no. 1, 297-311.

\bibitem{WC2} Wang, Zhou; Chen, Jianlong On two open problems about strongly clean rings. Bull. Austral. Math. Soc. 70 (2004), no. 2, 279-282.

\bibitem{WC} Wang, Zhou; Chen, Jianlong On strongly clean general rings. J. Math. Res. Exposition 27 (2007), no. 1, 28-34.

\bibitem{WC3} Wang, Zhou; Chen, Jianlong Pseudo Drazin inverses in associative rings and Banach algebras. Linear Algebra Appl. 437 (2012), no. 6, 1332-1345.

\bibitem{YC} Ying, Zhiling; Chen, Jianlong On quasipolar rings. Algebra Colloq. 19 (2012), no. 4, 683-692.


\end{thebibliography}
\end{document}